\documentclass{lmsMODIFIED}

\usepackage{amsmath, amscd, amsfonts, amssymb, graphicx}

\newtheorem{theorem}{Theorem}[section] 
\newtheorem{lemma}[theorem]{Lemma}     

\newtheorem{proposition}[theorem]{Proposition}

\newnumbered{definition}[theorem]{Definition}
\newnumbered{remark}[theorem]{Remark}
\newnumbered{example}[theorem]{Example}
\newnumbered{conjecture}[theorem]{Conjecture}

\newcommand{\abs}[1]{\vert#1\vert}

\newcommand{\roi}{\mathcal{O}}
\newcommand{\res}[1]{\overline{#1}}
\newcommand{\Char}[1]{\mbox{char}_{#1}}

\newcommand{\al}{\alpha}
\newcommand{\mult}[1]{#1^{\times}}
\newcommand{\comment}[1]{}
\newcommand{\sub}[1]{{\mbox{{\scriptsize {#1}}}}}

\renewcommand{\cal}[1]{\mathcal{#1}}
\renewcommand{\frak}[1]{\mathfrak{#1}}
\newcommand{\bb}[1]{\mathbb{#1}}

\title[Euler characteristics, Fubini's theorem]{Euler characteristics, Fubini's theorem, and the Riemann-Hurwitz formula}

\author{Matthew Morrow}

\classno{14E22 (primary), 03C60, 14G22, 14C17 (secondary)}

\extraline{The author is supported by an EPSRC Doctoral Training Grant at the University of Nottingham}

\begin{document}

\maketitle

\begin{abstract}
We relate Fubini's theorem for Euler characteristics to Riemann-Hurwtiz formulae, and reprove a classical result of Iversen. The techniques used include algebraic geometry, complex geometry, and model theory. Possible applications to the study of wild ramification in finite characteristic are discussed.
\end{abstract}

\section*{Introduction}
The first section of the paper reviews the concept of an Euler characteristic for a first order structure in model theory. The discussion is purely algebraic for the benefit of readers unfamiliar with model theory, and various examples are given.

Once an Euler characteristic is interpreted as an integral, it is natural to ask whether Fubini's theorem holds; that is, whether the order of integration can be interchanged in a repeated integral. In the second section we consider finite morphisms between smooth curves over any algebraically closed field, and show that Fubini's theorem is almost equivalent to the Riemann-Hurwitz formula. More precisely, in characteristic zero the two are equivalent and so Fubini's theorem is satisfied, whereas in finite characteristic the possible presence of wild ramification implies that, for any Euler characteristic, interchanging the order of integration is not always permitted.

The third section discusses a notion weaker than the full Fubini property (a so-called strong Euler characteristic \cite{Krajicek} \cite{Krajicek_Scanlon}), but which is sufficent for our applications. We show that over an algebraically closed field of characteristic zero, there is exactly one strong Euler characteristic (over the complex numbers, this is the usual topological Euler characteristic).

We return to finite morphisms between algebraic varieties, this time considering surfaces. Again, Fubini's theorem is related to a Riemann-Hurwitz formula, originally due to Iversen \cite{Iversen}. Our methods provide a new proof of his result.

The paper finishes with a discussion of the geometric approach to ramification theory of local fields.

\subsection*{Acknowledgments}
The conference 'Motivic Integration and its Interactions with Model Theory and Non-Archimedean Geometry' at the ICMS, Edinburgh, during May 2008 encouraged me to think about these ideas. Part of this text was written while visiting the IHES, and I am grateful for the excellent working environment which this provided. This visit would not have been possible without the generosity of the Cecil King Foundation and the London Mathematical Society, in the form of the Cecil King Travel Scholarship.

I thank my supervisor I. Fesenko for his constant encouragement.

\section{Structures and Euler characteristics}
First we present some elementary objects from model theory from a perspective suitable for this work. We must understand what sort of sets we can measure and what it means to measure them. This material is well-known but hopefully this explicit exposition will appeal to those unfamiliar with the theory.

\subsection{Structures on a field}
Given a set $\Omega$, a {\em ring} of subsets of $\Omega$ is defined to be a non-empty collection of subsets $\cal{R}$ of $\Omega$ such that \[A,B\in\Omega\Rightarrow A\setminus B,\,A\cup B,\,A\cap B\in\Omega.\] It is enough to assume that $\cal{R}$ is closed under differences and unions for this implies it is closed under intersections. A ring of sets is said to be an {\em algebra} if and only if it contains $\Omega$.

Following van den Dries \cite{vandenDries} we define a {\em structure} $\cal{A}=(\cal{A}(\Omega^n))_{n=0}^{\infty}$ on $\Omega$ to be an algebra $\cal{A}(\Omega^n)$ of sets on $\Omega_n$ for each $n\ge0$ such that
\begin{enumerate}
\item if $A\in\cal{A}(\Omega^n)$ then $A\times\Omega,\Omega\times A\in\cal{A}(\Omega^{n+1})$;
\item $\{(x_1,\dots,x_n)\in\Omega^n:x_1=x_n\}\in\cal{A}(\Omega^n)$;
\item if $\pi:\Omega^{n+1}\to\Omega^n$ is the projection map to the first $n$ coordinates, then $A\in\cal{A}(\Omega^{n+1})$ implies $\pi(A)\in\cal{A}(\Omega^n)$.
\end{enumerate}
Given a structure, one refers to the sets in $\cal{A}(\Omega^n)$ as being the {\em definable} subsets of $\Omega^n$. If $A\subseteq\Omega^n$ and $f: A\to\Omega^m$ then $f$ is said to be definable if and only if its graph belongs to $\cal{A}(\Omega^{n+m})$.

\begin{proposition}
Let $\cal{A}$ be a structure on a set $\Omega$. Then
\begin{enumerate}
\item if $A\in\cal{A}(\Omega^n),B\in\cal{A}(\Omega^m)$ then $A\times B\in\cal{A}(\Omega^{n+m})$;
\item if $1\le i<j\le n$, then $\{(x_1,\dots,x_n)\in\Omega^n:x_i=x_j\}$ is in $\cal{A}(\Omega^n)$;
\item if $\sigma$ is a permutation of $\{1,\dots,n\}$, then the function $\Omega^n\to\Omega^n$ given by permuting the indices of the coordinates by $\sigma$ is definable.
\end{enumerate}
Moreover, if $A\subseteq\Omega^n$ and $f:A\mapsto\Omega^m$ is definable, then
\begin{enumerate}
\item $A$ is definable;
\item if $B\subseteq A$ is definable, then $f(B)$ is definable, and the function given by restricting $f$ to $B$ is definable;
\item if $B\in\cal{A}(\Omega^m)$, then $f^{-1}(B)\in\cal{A}(\Omega^n)$;
\item if $f$ is injective, then its inverse is definable;
\item if $B\supseteq f(A)$ and $g:B\to\Omega^l$ is definable, then $g\circ f:A\to\Omega^l$ is definable.
\end{enumerate}
\end{proposition}
\begin{proof}
These are straightforward to check; proofs may be found in \cite{vandenDries}.
\end{proof}

\begin{remark}
If $\cal{L}$ is a first order language of logic, and $\Omega$ is an $L$-structure, then there is a structure on $\Omega$ in which $\cal{A}(\Omega^n)$ consists of precisely those sets of the form \[\{x\in\Omega^n:\Omega\models\phi(x,b)\}\] where $\phi(x,y)$ is formula of $L$ in variables $x_1,\dots,x_n,y_1,\dots,y_m$ and $b\in\Omega^m$; that is, those sets which are definable with parameters in the sense of model theory.

Realistically, any structure in which we will be interested will arise in this way as the parameter-definable sets of some language. But for the reader less familiar with logic, the axiomatic approach above is more immediately appealing, though ultimately less satisfying.
\end{remark}

\begin{example}\label{example_structures} We present some examples to explain what we can and cannot study using structures. All are well-known.
\begin{enumerate}
\item If $\Omega$ is an arbitrary set, we may take $\cal{A}(\Omega^n)$ to be collection of all subsets of $\Omega^n$; that is, every set is definable.
\item If $k$ is an algebraically closed field, let $\cal{A}(k^n)$ be the ring of sets generated by the Zariski closed subsets of $k^n$; such sets are called constructible. It is known that $(\cal{A}(k^n))_n$ forms a structure on $k$. The difficulty is establishing that such sets are closed under projection; this may either be proved in a model theoretic setting, where it is equivalent to establishing that the theory of algebraically closed fields admits quantifier elimination, or it may be seen as a special case of a result of algebraic geometry concerning constructible subsets of Noetherian schemes (see e.g. \cite{Hartshorne} exercises 3.17-3.19).
\item If $k$ is an arbitrary field, recall that an affine subset of $k^n$ is a set of the form $a+X$ where $a\in k^n$ and $X$ is a $k$-subspace of $k^n$. Letting $\cal{A}(k^n)$ be the ring of sets generated by affine subsets of $k^n$ gives a structure on $k$.
\item If $\bb{R}$ is the real line, then let $\cal{A}(\bb{R}^n)$ be the ring of sets generated by $\{x\in\bb{R}^n: p(x)\ge 0\}$ for $p\in\bb{R}[X_1,\dots,X_n]$; the sets in $\cal{A}(\bb{R}^n)$ are called semi-algebraic subsets of $\mathbb{R}^n$. This gives a structure for $\bb{R}$. Again, the difficulty is verifying that such sets are closed under projection.
\item None of the following give structures on the real line: the Borel sets, the Lebesgue measurable sets, the Souslin sets.
\end{enumerate}
So structures are typically quite coarse from the point of view of classical analysis and measure theory.
\end{example}

\subsection{Euler characteristics and the Grothendieck ring of a structure}
Having introduced the sets of interest we now discuss what it means to take the measure of such a set.

\begin{definition}
Let $\Omega$ be a set with a structure $\cal{A}$. An {\em Euler characteristic} is a map $\chi$ from the definable sets to some commutative ring $R$, i.e. \[\chi:\bigsqcup_{n=0}^{\infty}\cal{A}(\Omega^n)\to R,\] which satisfies
\begin{enumerate}
\item if $A,B\in\cal{A}(\Omega^n)$ are disjoint, then $\chi(A\sqcup B)=\chi(A)+\chi(B)$;
\item if $A\in\cal{A}(\Omega^n)$, $B\in\cal{A}(\Omega^m)$, then $\chi(A\times B)=\chi(A)\chi(B)$;
\item if $A\in\cal{A}(\Omega^n)$, $B\in\cal{A}(\Omega^m)$ and there is a definable bijection $f:A\to B$, then $\chi(A)=\chi(B)$.
\end{enumerate}
\end{definition}

\begin{remark}
From the additivity of $\chi$, one might think that an Euler characteristic is similar to a measure in the classical sense. The vast difference between the two is the invariance of $\chi$ under definable bijections. For example, if $\Omega$ is a field $k$, $\al\in\mult{k}$, and multiplication by $\al$ is a definable map from $k$ to itself, then $\chi(A)=\chi(\al A)$ for definable $A\subseteq k$; in other words, scaling a set does not affect its size. Or if $\Omega$ is the real line and $x\mapsto x^2$ is definable, then for any definable $A$ of the positive reals, $\chi(\{x^2:x\in A\})=\chi(A)$.

Some authors prefer the term {\em generalised Euler characteristic} or {\em additive invariant}, to avoid possible confusion with the {\em topological Euler characteristic} $\chi_\sub{top}$ for complex projective manifolds, defined as the alternating sum of the Betti numbers.
\end{remark}

\begin{example} The easiest example of an Euler characteristic is counting measure: let $\Omega$ be a finite set, $\cal{A}(\Omega^n)$ the algebra of all subsets of $\Omega^n$, and set $\chi(A)=\abs{A}$ to define a $\bb{Z}$-valued Euler characteristic.

Explicitly exhibiting more interesting Euler characteristics requires some work, so we present here without proof some known examples using the structures of example \ref{example_structures}.
\begin{enumerate}
\item Let $k$ be a field, equipped with the structure generated by the affine subsets. If $k$ is infinite then there is a unique $\bb{Z}[t]$-valued Euler characteristic $\chi$ which satisfies \[\chi(a+X)=t^{\dim_k X}\] where $a\in k^n$ and $X$ is a $k$-subspace of $k^n$.
\item Give $\bb{R}$ the structure of semi-algebraic sets. Then there is a unique $\bb{Z}$-valued Euler characteristic $\chi$ which satisfies \[\chi((0,1))=-1.\]
\item Give $\bb{C}$ the structure of constructible sets; then there is a unique Euler characteristic $\chi_\sub{top}$ which agrees with the topological Euler characteristic for any projective manifold.
\end{enumerate}
\end{example}

\begin{definition}
Let $\Omega$ be a set with structure $\cal{A}$. The associated {\em Grothendieck ring}, denoted $K_0(\Omega)$ (though it does of course depend on the structure, not just the set $\Omega$), is defined to be the free commutative unital ring generated by symbols $[A]$ for $A$ a definable subset of $\Omega^n$, any $n\ge 0$, modulo the following relations
\begin{enumerate}
\item if $A,B\in\cal{A}(\Omega^n)$ are disjoint, then $[A\sqcup B]=[A]+[B]$;
\item if $A\in\cal{A}(\Omega^n)$, $B\in\cal{A}(\Omega^m)$, then $[A\times B]=[A][B]$;
\item if $A\in\cal{A}(\Omega^n)$, $B\in\cal{A}(\Omega^m)$ and there is a definable bijection $f:A\to B$, then $[A]=[B]$.
\end{enumerate}
\end{definition}

\begin{remark}
The map $A\mapsto [A]$ defines a $K_0(\Omega)$-valued Euler characteristic on $\Omega$, which is universal in the sense that if $\chi:\bigsqcup_{n=0}^{\infty}\cal{A}_n\to R$ is an Euler characteristic, then there is a unique ring homomorphism $\chi':K_0(\Omega)\to R$ such that $\chi(A)=\chi'([A])$ for any definable $A$. Thus $A\mapsto[A]$ is the most general Euler characteristic of a structure.

Note that if $\{x\}\subseteq\Omega^n$ is a single point, and $A\subseteq\Omega^m$ is definable, then projection induces a definable isomorphism $\{x\}\times A\to A$. So $[\{x\}][A]=[A]$ for all definable $A$ and therefore $[\{x\}]$=1; more generally, $[B]=\abs{B}$ for any finite definable set $B$.
\end{remark}

\begin{remark}\label{remark_Euler_characteristic_of_varieties}
{\em Extending the Euler characteristic to varieties.} Assume that $\Omega=k$ is an algebraically-closed field with the structure $\cal{A}$ of constructible subsets. Let $V$ be a separated algebraic variety over $k$ - our varieties shall usually consist only of the closed points of the corresponding scheme -  and let $\cal{A}(V)$ be the ring generated by the Zariski closed subsets of $V$, i.e. the algebra of constructible subsets of $V$.

It is straightforward to prove that $\chi$ uniquely extends to $\cal{A}(V)$ in such a way that if $U\subseteq V$ is an affine open or closed subset, $C\subseteq U$ is constructible, and $i:U\to\bb{A}_k^d$ is an open or closed embedding for some $d$, then $\chi(C)=\chi(i(C))$.
\end{remark}

\begin{remark}\label{remark_measure_to_integral}
{\em Extending the measure to an integral.} If $\Omega$ is a set equipped with a structure and Euler characteristic $\chi$, then there is a unique $R$-linear map $\int\,d\chi$from the space of functions spanned by characteristic functions of definable sets to $R$ which satisfies $\int\Char{A}\,d\chi=\chi(A)$ for any definable $A$. We will allow ourselves to use typical notation for integrals, writing $\int f(x)\,d\chi(x)$.
\end{remark}

\section{Riemann-Hurwitz and Fubini's theorem for curves}
Here we relate Fubini's theorem for Euler characteristics to the Riemann-Hurwitz formula for morphisms between curves; then we produce a startling result implying that in finite characteristic it is always possible for Fubini's theorem to fail.

Throughout this section $k$ is an algebraically closed field of arbitrary characteristic, $\cal{A}$ is the structure of constructible sets, and $\chi$ is a fixed $R$-valued Euler characteristic on $\cal{A}$. By a curve $C$ over $k$, in this section, we mean a smooth, one-dimensional, irreducible algebraic variety over $k$; we only consider the closed points of $C$. Following remarks \ref{remark_Euler_characteristic_of_varieties} and \ref{remark_measure_to_integral} the space of integrable functions on $C$ is the $R$-module generate by characteristic functions of constructible sets; the integral on this space will be denoted $\int_C\cdot\,d\chi$.

Let $\phi:C_1\to C_2$ be a non-constant morphism of curves. We will study whether Fubini's theorem holds for the morphism $\phi$, which is to say that for each $y\in C_2$, the fibre $\phi^{-1}(y)$ is constructible, that $y\mapsto\chi(\phi^{-1}(y))$ is integrable, and finally that $\chi(C_1)=\int_{C_2}\chi(\phi^{-1}(y))\,d\chi(y)$. The problem immediately simplifies:

\begin{lemma}\label{lemma_Fubini_formula_for_curves}
Fubini's theorem holds for a separable morphism $\phi:C_1\to C_2$ of projective curves if and only if the following formula relating the Euler characteristics of $C_1$ and $C_2$ is satisfied: \[\chi(C_1)=\chi(C_2)\deg\phi-\sum_{x\in C_1}(e_x(\phi)-1),\] where $e_x(\phi)$ is the ramification degree of $\phi$ at $x$.
\end{lemma}
\begin{proof}
Let $\Sigma\subseteq C_1$ be the finite set of points at which $\phi$ is ramified. Let $y$ be a point of $C_2$. The fibre $\phi^{-1}(y)$ is finite; moreover, it contains exactly $\deg\phi$ points when $y\notin\phi(\Sigma)$. So each fibre is certainly constructible and $\chi(\phi^{-1}(y))=\abs{\phi^{-1}(y)}$. Thus $y\mapsto\chi(\phi^{-1}(y))$ is constant off the finite set $\phi(\Sigma)$ and hence is integrable on $C_2$; integrating obtains \[\int_{C_2}\chi(\phi^{-1}(y))\,d\chi(y)=\chi(C_2\setminus\phi(\Sigma))\deg\phi+\sum_{y\in\phi(\Sigma)}\abs{\phi^{-1}(y)}.\] The fundamental ramification equality $\sum_{x\in\phi^{-1}(y)}e_x(\phi)=\deg\phi$ transforms this into \[\chi(C_2)\deg\phi-\sum_{y\in\phi(\Sigma)}\sum_{x\in\phi^{-1}(y)}(e_x(\phi)-1),\] which completes the proof.
\end{proof}

\begin{remark}
More generally, if $\mbox{char}\,k=p>0$ and $\phi:C_1\to C_2$ is a morphism of projective curves which is not necessarily separable, then we decompose $\phi$ as $\phi=\phi_{\mbox{{\scriptsize sep}}}\circ F^m$; here $F$ is the Frobenius morphism of $C_1$, $\phi_{\mbox{{\scriptsize sep}}}:C_1\to C_2$ is a separable morphism, and $m$ is a non-negative integer. The previous proof shows that Fubini holds for $\phi$ if and only if \[\chi(C_1)=\chi(C_2)\deg\phi_{\mbox{{\scriptsize sep}}}-\sum_{x\in C_1}(e_x(\phi_{\mbox{{\scriptsize sep}}})-1).\] So Fubini holds for $\phi$ if and only if it holds for the separable part $\phi_{\mbox{{\scriptsize sep}}}$; in particular, Fubini holds for any purely inseparable morphism of projective curves

For this reason we are justified in focusing our attention on separable morphisms.
\end{remark}

\begin{remark}
More usually Fubini's theorem is concerned with measuring subsets of product space via repeated integrals; let us show that this is the same as our current activity considering fibres of morphisms between projective curves.

Suppose $\phi:C_1\to C_2$ is a separable morphism of projective curves over $k$. Then $\phi$ is a finite morphism, so that if $U_2\subseteq C_2$ is a non-empty, affine, open subset then the same is true of $U_1=\phi^{-1}(U_2)$. Choose closed embeddings $U_1\to\bb{A}_k^n$, $U_2\to\bb{A}_k^m$ and let $\Gamma=\{(x,\phi(x))\in\bb{A}_k^{n+m}:\,x\in U_1\}$ be the graph of $\phi_{U_1}$.

It is immediate that the integral $\int_{k^n}\int_{k^m}\Char{\Gamma}(x,y)\,d\chi(y)d\chi(x)$ is well-defined and equal to $\chi(U_1)$. Conversely, if we fix $y\in U_2$ then $\int_{k^n}\Char{\Gamma}(x,y)\,d\chi(x)=\chi(\phi^{-1}(y))$; arguing as in the previous lemma now obtains \[\int_{k^m}\int_{k^n}\Char{\Gamma}(x,y)\,d\chi(x)d\chi(y)=\chi(U_2)\deg\phi-\sum_{x\in U_1}(e_x(\phi)-1).\] So interchanging the order of integration preserves the value of the integral if and only if \[\chi(U_1)=\chi(U_2)\deg\phi-\sum_{x\in U_1}(e_x(\phi)-1).\]

Further, $C_2\setminus U_2$ and $\phi^{-1}(C_2\setminus U_2)=C_1\setminus U_1$ are finite sets and it is straightforward to verify, similarly to the previous lemma, that \[|C_1\setminus U_1|=|C_2\setminus U_2|\deg{\phi}-\sum_{x\in C_1\setminus U_1}(e_x(\phi)-1).\] Taking the sum of the previous two formulae shows that Fubini's theorem holds for $\phi:C_1\to C_2$ if and only if the the repeated integrals of $\Char{\Gamma}$ are equal.
\end{remark}

Recall that the Riemann-Hurwitz formula states that if $\phi:C_1\to C_2$ is a non-constant morphism of projective curves, then there are integers $\tilde{e}_x(\phi)$ for each $x\in C_1$ (which we shall call the {\em Riemann-Hurwitz ramification degrees}) such that $\tilde{e}_x(\phi)\ge e_x(\phi)$, with equality if and only if $\phi$ is tamely ramified at $x$, and such that \[2(1-g_2)=2(1-g_1)\deg\phi-\sum_{x\in C_1}(\tilde{e}_x(\phi)-1),\] where $g_i$ is the genus of $C_i$. It is apparent that Fubini's theorem and the Riemann-Hurwtiz formula are related.

\begin{remark}\label{remark_RH_ramification_degree_is_almost_discriminant}
The non-negative integer $\tilde{e}_x(\phi)-1$ is equal to the different of the extension $\roi_{{C_1},x}/\roi_{{C_2},\phi(x)}$ of discrete valuation rings, though we will not use this fact.
\end{remark}

\begin{remark}
It is useful to have some explicit examples of morphisms between projective curves. Let $f(t)$ be a polynomial over $k$ and let $\Gamma_f$ be the algebraic variety over $k$ which is the graph of $f$, i.e. \[\Gamma_f=\{(x,y)\in\bb{A}_k^2:y=f(x)\}.\] Let $F:\bb{A}_k^1\to\Gamma_f$ be the morphism $F(x)=(x,f(x))$ and let $\pi:\Gamma_f\to\bb{A}_k^1$ be the projection map $\pi(x,y)=y$. Note that $F$ is an isomorphism of algebraic varieties and that $\pi\circ F=f$; here we abuse notation and write $f$ for the morphism $\bb{A}_k^1\to\bb{A}_k^1$ induced by the polynomial $f(t)$. Let $\Gamma_f^*$ denote the projective closure of $\Gamma_f$, obtained by adding a single point at infinity. The morphisms $F,\pi,f$ extend to morphisms $F:\bb{P}_k^1\stackrel{\cong}{\to}\Gamma_f^*$, $\pi:\Gamma_f^*\to\bb{P}_k^1$, $f:\bb{P}_k^1\to\bb{P}_k^1$.

The previous remark implies that the following are all equivalent:
\begin{enumerate}
\item Fubini holds for $f:\bb{P}_k^1\to\bb{P}_k^1$;
\item Fubini holds for $f:\bb{A}_k^1\to\bb{A}_k^1$;
\item The repeated integrals of $\Char{\Gamma_f}$ are equal.
\end{enumerate}
\end{remark}

To make use of the examples afforded by the previous remark we now calculate the ramification degrees:

\begin{lemma}
We retain the notation of the previous remark. The ramification degrees of $f:\bb{P}_k^1\to\bb{P}_k^1$ are \[e_a(f)=\begin{cases}\nu_{t-a}(f(t)-f(a))& a\in k=\bb{A}_k^1\\ \deg f&a=\infty,\end{cases}\]
and the Riemann-Hurwitz ramification degrees are \[\tilde{e}_a(f)=\begin{cases}1+\nu_{t-a}(f'(t))& a\in k=\bb{A}_k^1 \\\deg f+(\deg f-\deg f'-1)&a=\infty.\end{cases}\]
\end{lemma}
Here $\nu_{t-a}$ denotes the $t-a$-adic valuation on $k(t)$.
\begin{proof}
The ramification degrees are clear so we only consider the Riemann-Hurwitz degrees.

Write $s=f(t)$ so that $f:\bb{P}_k^1\to\bb{P}_k^1$ corresponds to the inclusion of function fields $K(s)\to K(t)$. A local coordinate $t_a\in K(t)$ at $a\in k$ is $t-a$; a local coordinate $s_b\in K(s)$ at $b=f(a)$ is $s-b$. By definition of the Riemann-Hurwitz ramification degree, \[\tilde{e}_a(f)-1=\nu_{t-a}(\frac{d}{dt_a}{s_b});\] writing $f(t)-b=g(t-a)$ for some polynomial $g$ gives \[\nu_{t-a}(\frac{d}{dt_a}{s_b})=\nu_{t-a}(g'(t-a))=\nu_{t-a}(f'(t)).\]

Secondly, $f(\infty)=\infty$ and local parameters are given by $t^{-1}$, $s^{-1}$; therefore the Riemann-Hurwitz ramification degree at infinity is given by \[\tilde{e}_{\infty}(f)=\nu_{t^{-1}}\left(\frac{1}{f(t)}\right)+1=\deg f+(\deg f-\deg f'-1).\]
\end{proof}

\begin{example}
For any integer $m>1$ not divisible by $\mbox{char}\,k$, let $f(t)=t^m$ in the previous remark. Then $f:\bb{P}_k^1\to\bb{P}_k^1$ is unramified away from $0$ and infinity, with $e_0(f)=e_{\infty}(f)=m$. Thus Fubini's theorem holds for $f$ (or, equivalently, for the set $\Gamma_f\subseteq k\times k$) if and only if $\chi(\bb{P}_k^1)=m\chi(\bb{P}_k^1)-2(m-1)$; that is, if and only if $(\chi(\bb{P}_k^1)-2)(m-1)=0$.

However, now assume $\mbox{char}\,k=p>0$ and set $f(t)=t^p-t$. Then $f:\bb{P}_k^1\to\bb{P}_k^1$ is unramified outside infinity, where it is wildly ramified of degree $p$. Thus Fubini's theorem holds for $f$ (or, equivalently, for the set $\Gamma_f\subseteq k\times k$) if and only if $\chi(\bb{P}_k^1)=p\chi(\bb{P}_k^1)-(p-1)$; that is, if and only if $(\chi(\bb{P}_k^1)-1)(p-1)=0$.

Taking $m=p+1$ in the previous two paragraphs shows that Fubini fails for one of the sets $\Gamma_{X^p-X}$, $\Gamma_{X^p}$ or that $p$ is an idempotent in $R$.
\end{example}

The example shows that Fubini's theorem can fail when in finite characteristic:

\begin{theorem}\label{theorem_Fubini_fails_in_char_p}
Assume $\mbox{char}\,k=p>2$ and that $p\neq 1$ in $R$. Then there exists a subset of $k\times k$ for which Fubini's theorem does not hold.
\end{theorem}
\begin{proof}
If Fubini does hold for the sets $\Gamma_{X^{p+1}}$ and $\Gamma_{X^{p+2}}$ of the previous example then it follows that $\chi(\bb{P}_k^1)=2$. But then Fubini does not hold for $\Gamma_{X^p-X}$, unless $p-1=0$ in $R$.
\end{proof}

Now we prove the next main result, namely that Fubini's theorem forces $\chi$, our arbitrary Euler characteristic on the algebra of constructible sets, to be the usual Euler characteristic of a curve:

\begin{theorem}\label{theorem_euler_characteristic_is_usual_one}
Suppose that $\mbox{char}\,k\neq2$ and that Fubini's theorem is true for any non-constant, separable, tame morphism $\phi:C\to\bb{P}_k^1$ from a projective curve to the projective line. Then for any projective curve $C$ we have $\chi(C)=2(1-g)$, where $g$ is the genus of $C$.
\end{theorem}
\begin{proof}
For any integer $m>1$ not divisible by $\mbox{char}\,k$, the morphism $f:\bb{P}_k^1\to\bb{P}_k^1$ induced by $f(t)=t^m$ is separable and tame; therefore we may apply Fubini's theorem to deduce $(\chi(\bb{P}_k^1)-2)(m-1)=0$. Therefore $\chi(\bb{P}_k^1)=2$, which agrees with the desired genus formula.

Now let $C$ be a projective curve over $k$. By a classical result of algebraic geometry \cite[prop 8.1]{Fulton} there is, for any $n$ sufficiently large (depending on the genus $g$ of $C$), a non-constant morphism $\phi:C\to\bb{P}_k^1$ of degree $n$ with the property that any fibre contains at least $n-1$ points. For $n$ not divisible by $\mbox{char}\,k$ such a morphism is separable and tame; therefore we are permitted to apply Fubini's theorem, deducing \[\chi(C)=2\deg\phi-\sum_{x\in C}(e_x(\phi)-1).\]

But this is nothing other than the Riemann-Hurwitz formula for the morphism $\phi$; so we obtain $\chi(C)=2(1-g)$ as claimed.
\end{proof}

This allows us to strengthen the observation that Fubini fails in finite characteristic:

\begin{theorem}
Suppose that $\mbox{char}\,k\neq2$ and that Fubini's theorem is true for any non-constant, separable, tame morphism between projective curves. Then Fubini's theorem holds for a separable morphism between projective curves if and only if the morphism is tame.
\end{theorem}
\begin{proof}
The previous result implies that $\chi(C)=2(1-g)$ is the usual Euler characteristic of any projective curve $C$. Suppose that $\phi:C_1\to C_2$ is a separable morphism of projective curves which is not everywhere tame. Then the Riemann-Hurwitz formula tells us that \[\chi(C_1)=\chi(C_2)\deg{\phi}-\sum_{x\in C_1}(\tilde{e}_x(\phi)-1),\] which is incompatible with Fubini's theorem for $\phi$ as $\tilde{e}_x(\phi)\ge e_x(\phi)$ for all $x\in C_1$ with at least one value of $x$ for which we do not have equality.
\end{proof}

\begin{remark}\label{remark_wild_ramification}
More precisely, in the situation of the previous result, we have \[\chi(C_1)-\int_{C_2}\chi(\phi^{-1}(y))\,d\chi(y)=\sum_{x\in C_1}d_x(\phi),\] where $d_x(\phi)$ is defined by $\mathfrak{D}_x(\phi)=e_x(\phi)-1+d_x(\phi)$; here $\mathfrak{D}_x(\phi)$ denotes the different of the extension $\roi_{{C_1},x}/\roi_{{C_2},\phi(x)}$ of discrete valuation rings (see also remark \ref{remark_RH_ramification_degree_is_almost_discriminant}). $d_x(\phi)$ measures the wild ramification at $x$.

An Euler characteristic is typically considered an object of 'tame' mathematics \cite{vandenDries}, and so this formula is remarkable in that it expresses wild information purely in terms of tame.
\end{remark}

\begin{remark}
A two-dimensional local field is a complete discrete valuation field $F$ whose residue field $K$ is a usual local field. Such a field is not locally compact but a theory of integration on such spaces has been developed \cite{Fesenko-analysis-on-arithmetic-schemes}, \cite{Fesenko-analysis-on-loop-spaces}, \cite{Hrushovski-Kazhdan1}, \cite{Hrushovski-Kazhdan2}, \cite{Kim-Lee1}, \cite{Morrow_1}, \cite{Morrow_2}, \cite{Morrow_3}.

In \cite{Morrow_3} the author proved that the characteristic function of \[\Gamma=\{(x,y)\in F:\,(x,y-t^{-1}x^p)\in\cal{O}_F\times\cal{O}_F\}\] fails to satisfy Fubini's theorem; in fact, $\int_F\int_F \Char{\Gamma}(x,y)\,dxdy=0$ and $\int_F\int_F \Char{\Gamma}(x,y)\,dxdy=1$. This is similar phenomenon to what we have just observed for the Euler characteristic $\chi$.

These results suggest interpreting the Riemann-Hurwtiz formula as a modified 'repeated integral', adjusted in a suitable way to ensure that Fubini's theorem holds. Perhaps it is possible to modify the two-dimensional integration theory in a similar way.
\end{remark}

\section{Strong Euler characteristics}

In the previous section, we in fact only considered interchanging the order of integration in morphisms all of whose fibres were finite. This brief section is a study of the possible Euler characteristics which do satisfy this restricted version of Fubini's theorem.

\begin{definition}
Let $\Omega$ be a set with structure $\cal{A}$. An Euler characteristic $\chi$ is said to be {\em strong} if and only if whenever $f:A\to B$ is a definable function between two definable sets such that there exists a positive integer $n$, with $|\chi(f^{-1})(b)|=n$ for all $b\in B$, then $\chi(A)=n\chi(B)$.
\end{definition}

\begin{remark}
A strong Euler characteristic satisfies Fubini's theorem in a very weak sense. For suppose $\chi$ is an Euler characteristic, $A\subseteq \Omega^n$, $B\subseteq\Omega^m$ are definable, and $f:A\to B$ is an $n\mbox{-to-}1$ mapping as in the definition; set $\Gamma=\{(x,y)\in \Omega^n\times \Omega^m:x\in A,\,y\in B,\,f(x)=y\}$. Then Fubini's theorem holds for $\Char{\Gamma}$ if and only if $\chi(A)=n\chi(B)$.
\end{remark}

It is straightforward to establish non-existence in certain cases and uniqueness in others:

\begin{theorem}
Suppose $k$ is an algebraically closed field, of finite characteristic $>2$, with the structure of constructible sets; then no strong Euler characteristic exists.
\end{theorem}
\begin{proof}
This is just a restatement of theorem \ref{theorem_Fubini_fails_in_char_p}, where the counterexample did not require $\chi$ to satisfy the full Fubini property, but merely be strong.
\end{proof}

\begin{theorem}
Suppose $k$ is an algebraically closed field, of characteristic zero, with the structure of constructible sets; then at most one strong Euler characteristic exists, and it is $\bb{Z}$ valued.
\end{theorem}
\begin{proof}
Let $\chi_i$ be strong Euler characteristics, for $i=1,2$. The algebra of constructible subsets of $k^n$ is generated by the irreducible closed subsets, and therefore it is enough to establish $\chi_1(V)=\chi_2(V)$ for any irreducible closed $V\subseteq k^n$; this we do by induction on the dimension $d$ of $V$. Let $V'$ be the closure of $V$ in $\bb{P}_k^n$; then $V'\setminus V$ has dimension strictly less than that of $V$, and so, by the inductive hypothesis, it is enough to establish $\chi_1(V)=\chi_2(V')$.

Let $f:V'\to\bb{P}_k^d$ be a finite projective morphism; this always exists (see e.g. \cite[Lem. 6.4.27]{Liu}). Let $\Sigma\subset V'$ denote the points at which $V'$ is non-singular, or at which $f$ is not \'{e}tale; this is closed in $V'$ by \cite[Prop. 4.2.24, Cor. 4.4.12]{Liu}. Since morphisms of finite type are closed, $U:=\bb{P}_k^d\setminus f(\Sigma)$ is an open subset of $\bb{P}_k^d$, and it is non-empty because it contains the generic point (here it is important to observe that $K(V')/K(\bb{P}_k^1)$ is a {\em separable} extension of fields).

Hence the restriction of $f$ to $f^{-1}(U)$ is a finite \'{e}tale morphism to $\bb{P}_k^1$, i.e. an \'{e}tale cover, of degree $m=|K(V')/K(\bb{P}_k^1)|$; the assumption that each $\chi_i$ is strong implies \[\chi_i(f^{-1}(U))=m\chi_i(U)\] for $i=1,2$. Moreover, $\dim(V'\setminus f^{-1}(U))$ and $\dim(f(\Sigma))$ are both $<d$, and therefore the inductive hypothesis lets us deduce
\begin{align*}
\chi_1(V')&=\chi_1(f^{-1}(U))+\chi_1(V'\setminus f^{-1}(U))\\
	&=m(\chi_1(\bb{P}_k^d)-\chi_1(f(\Sigma)))+\chi_2(V'\setminus f^{-1}(U))\\
	&=m(\chi_1(\bb{P}_k^d)-\chi_2(f(\Sigma)))+\chi_2(V'\setminus f^{-1}(U))\\
	&=m(\chi_1(\bb{P}_k^d)-\chi_2(\bb{P}_k^d))+\chi_2(V').
\end{align*}
It remains only to prove that our two Euler characteristics agree on $\bb{P}_k^d$. Decomposing projective space into a disjoint union of constructible sets $\bb{P}_k^d=\bigsqcup_{i=0}^d\bb{A}_k^i$ and using multiplicativity of each $\chi_i$ on products, we have finally reduced the problem to proving that $\chi_1(\bb{A}_k^1)=\chi_2(\bb{A}_k^1)$.

But the argument of the first paragraph of theorem \ref{theorem_euler_characteristic_is_usual_one}, which is valid for any strong Euler characteristic, establishes that $\chi_i(\bb{A}_k^1)=1$ for $i=1,2$.
\end{proof}

\begin{remark}
If $k=\bb{C}$ then a strong Euler characteristic does exist on the structure of constructible sets, namely the topological Euler characteristic. This follows from the classical result that if $\tilde{X}\to X$ is an $n$-sheeted covering of a CW-complex $X$, then $\chi_\sub{top}(\tilde{X})=n\chi_\sub{top}(X)$.

The Lefschetz principle (i.e. that the first order theory of algebraically closed fields of characteristic zero is complete; see \cite{Cherlin} for a classical discussion of this principle) now implies that a strong Euler characteristic exists for any algebraically closed field of characteristic zero.
\end{remark}

\begin{remark}
The inclusion of this material is inspired by \cite{Krajicek} and \cite{Krajicek_Scanlon}, where strong Euler characteristics (in fact, the definition of 'strong' in these papers is slightly stronger than the definition given here) are discussed from the perspective of model theory. In \cite{Krajicek_Scanlon}, it is proved that a universal strong Euler characteristic $\mbox{Def}(k)\to K_0^\sub{s}(k)$ exists, and so our previous theorem and remark prove that if $k$ is algebraically closed of characteristic zero, then $K_0^\sub{s}(k)=\bb{Z}$.
\end{remark}

\section{Riemann-Hurwitz and Fubini's theorem for surfaces}
Now we generalise the results of the previous section from curves to surfaces. $k$ continues to be an algebraically closed field, and $\chi$ is a fixed $R$-valued Euler characteristic on the structure of constructible sets. In this section, 'surface' means a smooth, two-dimensional, irreducible algebraic variety over $k$, whereas a 'curve' is merely a one-dimensional, reduced, algebraic variety over $k$

If $\phi:S_1\to S_2$ is a finite morphism between projective surfaces of degree $n$, then let $B\subseteq S_2$ be the set of $y\in S_2$ such that $\phi^{-1}(y)$ does not contain $n$ points. Zariski's purity theorem (see e.g. \cite[ex. 8.2.15]{Liu} or \cite{Zariski}) states that $B$ is pure of dimension one; let $B_1,\dots,B_r$ be its irreducible components, and let $n_i$ be the degree of the morphism $\phi|_{\phi^{-1}(B_i)}:\phi^{-1}(B_i)\to B_i$ (note that the degree is well-defined, as the base curve is irreducible, though the covering curve $\phi^{-1}(B_i)$ may be reducible). Using this data we may prove an analogue of lemma \ref{lemma_Fubini_formula_for_curves}:

\begin{theorem}
Let $\phi:S_1\to S_2$ be a finite morphism between projective surfaces, with notation as in the previous paragraph. Then Fubini holds for $\phi$ (in the same sense as the previous section) if and only if the following formula relating $\chi(S_1)$ and $\chi(S_2)$ is satisfied: \[\chi(S_1)=\chi(S_2)\,\deg \phi-\sum_{i=1}^r(n-n_i)\chi(B_i)+\sum_{y\in B}\left(|\phi^{-1}(y)|-n+\sum_{i=1}^r(n-n_i)m_i(y)\right),\] where $m_i(y)$ denotes the number of local branches of $B_i$ at $x$. If $\chi$ is a strong Euler characteristic then this formula holds.
\end{theorem}
\begin{proof}
The normalisation of $B$ is by definition $\pi_B:\tilde{B}=\bigsqcup_{i=1}^r\tilde{B}_i\to B$, where $\pi_i:\tilde{B}_i\to B_i$ is the normalistion of the irreducible curve $B_i$. Write $D=\phi^{-1}(B)$, and let $\pi_D:\tilde{D}\to D$ be its normalisation in the same way as $B$; the functoriality of normalising implies that there is an induced morphism $\tilde{\phi}:\tilde{D}\to\tilde{B}$ such that $\pi_B\tilde{\phi}=\phi|_D\pi_D$. 

Let $Z\subset B$ be a large enough finite set of points such that $Z$ includes all singular points of the curve $B$, $\phi^{-1}(Z)$ includes all singular points of the curve $\phi^{-1}(B)$, and $\tilde{\phi}^{-1}(\pi_B^{-1}(Z))$ includes all points of ramification of $\tilde{\phi}$. Then $\pi_D$ and $\pi_B$ induce isomorphisms $\tilde{D}\setminus \tilde{\phi}^{-1}(\pi_B^{-1}((Z))\cong D\setminus \phi^{-1}(Z)$ and $\tilde{B}\setminus \pi_B^{-1}(Z)\cong B\setminus Z$; therefore
\begin{align*}
\int_{B\setminus Z}|\phi^{-1}(y)|\,d\chi(y)
	&=\int_{\tilde{B}\setminus \pi_B^{-1}(Z)}|\tilde{\phi}^{-1}(y)|\,d\chi(y)\\
	&=\int_{\tilde{B}}|\tilde{\phi}^{-1}(y)|\,d\chi(y)-\int_{\pi_B^{-1}(Z)}|\tilde{\phi}^{-1}(y)|\,d\chi(y)\\
	&=\sum_{i=1}^r\int_{\tilde{B}_i}|\tilde{\phi}^{-1}(y)|\,d\chi(y)-\sum_{y\in\pi_B^{-1}(Z)}|\tilde{\phi}^{-1}(y)|
\end{align*}
Further, as we saw in the proof of lemma \ref{lemma_Fubini_formula_for_curves},
\begin{align*}
\int_{\tilde{B}_i}|\tilde{\phi}^{-1}(y)|\,d\chi(y)
	&=n_i\chi(\tilde{B}_i)+\sum_{y\in\tilde{B}_i\cap\pi_B^{-1}(Z)}(|\tilde{\phi}^{-1}(y)|-n_i).
\end{align*}
Since $\tilde{B}_i\setminus \pi_B^{-1}(Z)\cap\tilde{B}_i\cong B_i\setminus Z\cap B_i$, we have $\chi(\tilde{B}_i)=\chi(B_i)+\sum_{y\in Z\cap B_i}(m_i(y)-1)$; combining the last few identities gives
\begin{align*}
\int_B|\phi^{-1}(y)|\,d\chi(y)
	&=\sum_i n_i\chi(B_i)+\sum_in_i\sum_{y\in B_i\cap Z}m_i(y)-\sum_i|B_i\cap Z|\\
	&\phantom{=}-\sum_i\sum_{y\in\tilde{B}_i\cap\pi_B^{-1}(Z)}n_i+\sum_{y\in Z}|\phi^{-1}(y)|.
\end{align*}

To complete the proof, combine this identity with
\begin{align*}
\int_{S_2}|\phi^{-1}(y)|\,d\chi(y)
	&=n\chi(S_2\setminus B)+\int_B|\phi^{-1}(y)|\,d\chi(y)\\
	&=n\chi(S_2)-n(\sum_i\chi(B_i)-\sum_{y\in Z}(c(y)-1))+\int_B|\phi^{-1}(y)|\,d\chi(y),
\end{align*}
where $c(y)$ denotes the number of irreducible components of $B$ which pass through $y$ (note that $\sum_{y\in Z} c(y)=\sum_i|B_i\cap Z|$).
\end{proof}

\begin{remark}
When $k=\bb{C}$ and $\chi=\chi_\sub{top}$ is the topological Euler characteristic, which we have remarked earlier is a strong Euler characteristic, then the theorem proves that \[\chi_\sub{top}(S_1)=\chi_\sub{top}(S_2)\,\deg \phi-\sum_{i=1}^r(n-n_i)\chi(B_i)+\sum_{y\in B}\left(|\phi^{-1}(y)|-n+\sum_{i=1}^r(n-n_i)m_i(y)\right).\] The Lefschetz principle now implies that the formula remains true if we replace $k$ by any algebraically closed field of characteristic zero, and $\chi_\sub{top}(S_i)$ by the $l$-adic Euler characteristic (=alternating sum of Betti numbers of $l$-adic \'{e}tale cohomology of $S_i$, =degree of the second Chern class of $S_i$).

This generalisation of the Riemann-Hurwitz formula to surfaces is due to B. Iversen \cite{Iversen}, who established it with purely algebraic techniques by studying pencils of curves on the surfaces. Iversen remarks in his paper that a more topological proof should be possible when $k=\bb{C}$, and our approach provides that.
\end{remark}

\begin{remark}
A natural question now to ask is whether an analogue of the theorem holds in higher dimensions. If $X_1\to X_2$ is a finite morphism between $d$-dimensional smooth projective varieties over $k$, then the branch locus will be pure of dimension $d-1$, so one can hope to obtain results by induction on dimension. The difficulty which appears when the branch local has dimension $>1$ is that there is no functorial way to desingularise. It is unclear to the author at present how significant a problem this is.
\end{remark}

\begin{remark}
Another interesting question concerns the situation in characteristic $p$. We noted in remark \ref{remark_wild_ramification} that, for curves, the difference between the Euler characteristic and the integral over the fibres was a measure of the wild ramification. For surfaces, the situation is more complex, since the wild ramification of surfaces is not fully understood. However, assuming that there is no ferocious ramification present (this is when inseparable morphisms between curves appear), I.~Zhukov \cite{Zhukov-3} has successfully generalised Iversen's formula by defining appropriate ramification invariants; this provides an explicit formula for \[\chi(S_1)-\int_{S_2}|\phi^{-1}(y)|\,dy\] in terms of the wild ramification of the cover.
\end{remark}

\section{Ramification of local fields}

The most interesting related problem is the local situation; that is, ramification of local fields. Fix a complete discrete valuation field $F$ with perfect residue field $\res{F}$, and let $F^\sub{al}$ denote its algebraic closure.  Fix a finite Galois extension $L/F$ with Galois group $G$, and define the usual ramification objects as follows:
\begin{align*}
i_{L/F}(\sigma)&=\min\{\nu_F(\sigma(x)-x):x\in\roi_F\},\\
G_a&=\{\sigma\in G:i_{L/F}(\sigma)\ge a+1\}\quad (a\ge -1),\\
\eta_{L/F}(a)&=e_{L/F}^{-1}\int_0^a|G_x|\,dx\quad (a\ge -1)\\
	&=-1+e_{L/F}^{-1}\sum_{\sigma\in G}\min\{i_{L/F}(\sigma),a+1\}.
\end{align*}
One proves that $\eta_{L/F}$ is a strictly increasing, piecewise linear, function $[-1,\infty)\to [-1,\infty)$, and defines the {\em Hasse-Herbrand} function $\psi_{L/F}:[-1,\infty)\to [-1,\infty)$ to be its inverse. Upper ramification on the Galois group is defined by $G^a=G_{\psi_{L/F}(a)}$.

We now explain a geometric interpretation of these formulae. Since $\res{F}$ is perfect, $\roi_L/\roi_F$ is monogeneic; let $\xi$ be a chosen generator, with minimal polynomial $f\in\roi_F[X]$. Extend $\nu_F$ to all of $F^\sub{al}$ to give a $\bb{Q}$-valued valuation; we will write \[\frak{p}_{F^\sub{al}}^a=\{x\in F^\sub{al}:\nu_F(x)\ge a\}\] for any real number $a$ to denote the closed ball of radius $a$. By some rigid geometry, model theory, or explicit calculations, it is known that $f^{-1}(\frak{p}_{F^\sub{al}}^a)$ may be written in a unique way as a disjoint union of closed balls. Let $\pi_0(f^{-1}(\frak{p}_{F^\sub{al}}^a))$ denote this set of balls, and note that $G$ acts on it transitively since each ball contains at least one root of $f$.

\begin{lemma}
For $a\ge -1$, $\sigma\in G$ acts trivially on $\pi_0(f^{-1}(\frak{p}_{F^\sub{al}}^{\eta_{L/F}(a)+1}))$ if and only if $\sigma\in G_a$.
\end{lemma}
\begin{proof}
A nice sketch of this is given in \cite{Xiao_1}.
\end{proof}

So, for any $a\ge -1$, the kernel of the action of $G$ on $\pi_0(f^{-1}(\frak{p}_{F^\sub{al}}^{a+1}))$ is $G^a$. Further, the definition of the Hasse-Herbrand function implies that \[\frac{d\psi_{L/F}}{da}(a)=e_{L/F}^{-1}|G^a|^{-1},\] at least away from the ramification breaks, and therefore that \[\psi_{L/F}(a)=e_{L/F}^{-1}\int_{-1}^a|G^x|^{-1}\,dx-1,\] since both sides are $=-1$ at $a=-1$. But $|G:G^x|=|\pi_0(f^{-1}(\frak{p}_{F^\sub{al}}^{x+1}))|$, and so \[\psi_{L/F}(a)=f_{L/F}^{-1}\int_0^{a+1}|\pi_0(f^{-1}(\frak{p}_{F^\sub{al}}^x))|\,dx-1\tag{$\ast$}\] for all $a\ge -1$.

If we think of ``the number of connected components'' as a measure, then ($\ast$) is a repeated integral taken over certain fibres. This geometric approach to ramification is used by A.~Abbes and T.~Saito \cite{Abbes_Saito_1} \cite{Abbes_Saito_2} to develop ramification theory for complete discrete valuation fields with imperfect residue field, using rigid geometry. L.~Xiao has written a good overview \cite{Xiao_1} of their theory and established integrality of various conductors \cite{Xiao_2} \cite{Xiao_3}. I would not have been able to give the discussion above without the help of G.~Yamashita. Perhaps it is possible to replace the rigid geometry techniques by model theory, using the rich model theoretic structure of algebraically closed valued field, just as we have explored ramification of surfaces using the model theory of algebraically closed fields.

\affiliationone{
	Matthew Morrow\\
	Maths and Physics Building,\\
	University of Nottingham,\\
	University Park,\\
	Nottingham\\
	NG7 2RD\\
	United Kingdom\\
   	\email{matthew.morrow@maths.nottingham.ac.uk}}
\end{document}